\documentclass[11pt,a4paper]{amsart}
\usepackage{amsmath,amsfonts,amssymb,amsthm}
\usepackage{mathrsfs}

\usepackage[left=1.0in,right=1.0in,top=1.0in,bottom=1.0in]{geometry}

\newtheorem{thm}{Theorem}[section]
\newtheorem{cor}[thm]{Corollary}
\newtheorem{prop}[thm]{Proposition}
\newtheorem{lemma}[thm]{Lemma}

\theoremstyle{definition}
\newtheorem{Def}[thm]{Definition}

\newtheorem{remark}[thm]{Remark}
\newtheorem{notation}[thm]{Notation}

\theoremstyle{remark}
\newtheorem{claim}{Claim}

\newcommand{\T}{{\mathbb T}}
\newcommand{\N}{{\mathbb N}}
\newcommand{\Z}{{\mathbb Z}}
\newcommand{\B}{{\mathscr B}}
\newcommand{\E}{{\mathscr E}}
\newcommand{\TT}{{\mathfrak T}}
\newcommand{\cont}{\mathfrak{c}}
\newcommand\n{\mathfrak{n}}
\newcommand{\Zar}{{\mathfrak Z}}

\newcommand{\du}[1]{\widehat{#1}}
\newcommand{\round}[1]{almost ${#1}$-torsion}
\newcommand{\Round}[1]{Almost ${#1}$-torsion}
\newcommand\CL{\mathrm{Cl}}
\newcommand{\Mono}{{\mathrm{Mono}}}
\newcommand{\Hom}{{\mathrm{Hom}}}
\newcommand{\restr}[1]{\restriction_{{#1}}}

\begin{document}

\title{A Kronecker-Weyl theorem for subsets of abelian groups}

\author{Dikran Dikranjan}
\address[D. Dikranjan]{Dipartimento di Matematica e Informatica\\
Universit\`{a} di Udine\\
Via delle Scienze  206, 33100 Udine\\
Italy}
\email{dikran.dikranjan@uniud.it}

\author{Dmitri Shakhmatov}
\address[D. Shakhmatov]{Division of Mathematics, Physics and Earth Sciences\\
Graduate School of Science and Engineering\\
Ehime University\\
Matsuyama 790-8577\\
Japan}
\email{dmitri.shakhmatov@ehime-u.ac.jp}

\thanks{The first named author was partially supported by SRA, grants P1-0292-0101 and J1-9643-0101.}

\thanks{The second named author was partially supported by the Grant-in-Aid for Scientific Research (C) No.~22540089 by the Japan Society for the Promotion of Science (JSPS)}

\keywords{uniform distribution, Weyl criterion, Zariski closure, Zariski topology, potentially dense set, Markov problem, diophantine approximation, discrete flow, Bohr topology, Bohr compactification}

\subjclass[2000]{Primary: 20K30; Secondary: 03E15, 11K36, 11K60, 22A05, 37B05, 54D65, 54E52}

\begin{abstract}
Let $\N$ be the set of non-negative integer numbers, $\T$ the circle group and $\cont$ the cardinality of the continuum. Given an abelian group $G$  of size at most $2^\cont$ and a countable family $\E$ of  infinite subsets of $G$,  we construct ``Baire many'' monomorphisms $\pi: G\to \T^\cont$ such that $\pi(E)$ is dense in $\{y\in \T^\cont: ny=0\}$ whenever  $n\in\N$, $E\in \E$, $nE=\{0\}$  and $\{x\in E: mx=g\}$ is finite for all $g\in G$ and $m$ such that $n=mk$
for some $k\in\N\setminus\{1\}$. We apply this result to obtain an algebraic description of countable potentially dense subsets of abelian groups, thereby making a significant progress towards a solution  of a problem of Markov going back to 1944. A particular case of our result yields  a positive answer to a problem of Tkachenko and Yaschenko \cite[Problem 6.5]{TY}. Applications to group actions and discrete flows on $\T^\cont$, diophantine approximation,  Bohr topologies and Bohr compactifications are also provided.
\end{abstract}

\maketitle

\setlength{\baselineskip}{14pt}

We refer the reader to \cite{Fu} for a background on abelian groups.
All undefined topological terms can be found in \cite{Eng}. 

We use $\N$ and $\N^+$ to denote the set of all natural numbers and positive natural numbers, respectively. The groups of integer numbers and real numbers 
are denoted by $\Z$ and $\mathbb{R}$, respectively.
We use $\mathbb{T}=\mathbb{R}/\Z$ to denote the circle group (written additively).
As usual, 
the symbol $|X|$ stands for the cardinality of a set $X$, and
we let
$\omega=|\N|$ and $\cont=|\mathbb{R}|$. 

Let $G$ be an abelian group. For every $m\in\N$, define $mG=\{mx:x\in G\}$ and $G[m]=\{x\in G:mx=0\}$. 
We say that $G$ is {\em bounded\/} if $G=G[n]$ for some $n\in\N^+$, and the 
minimal such $n$ is called the {\em exponent\/} of $G$. If $nG=G$ for every $n\in\N^+$, then $G$ is said to be {\it divisible\/}.  We denote by $r_0(G)$ the free rank of the group $G$ and by  $r_p(G)$ the $p$-rank of $G$ for a prime number $p$.  For 
a compact Hausdorff abelian group $K$, we use $\Hom(G,K)$ to denote the set of all group homomorphisms from $G$ to $K$ equipped with the topology of pointwise convergence, i.e., with the subspace topology that $\Hom(G,K)$ inherits from the Tychonoff product topology on $K^G$. Since $\Hom(G,K)$ is closed in the compact space $K^G$,  $\Hom(G,K)$ 
is a compact Hausdorff (abelian) group, with pointwise addition of homomorphisms as the group operation.   The set of all monomorphisms from $G$ to $K$ is denoted by $\Mono(G,K)$. Recall that $\du{G}=\Hom(G,\T)$ is the Pontryagin dual of  (the discrete abelian group)  $G$.

A topological group is called {\em precompact\/} (or {\em totally bounded\/}) if its completion is compact \cite{Weil}.

\section{Introduction}

Let $\mathbb{C}$ denote the complex plane and  $\mathbb{S}=\{z\in \mathbb{C}: |z|=1\}$. 
For an abelian group $K$, a map $\chi:K \to \mathbb{S}$ will be called an {\em $\mathbb{S}$-character\/} of $K$ provided that $\chi(0)=1$ and $\chi(x+y)=\chi(x) \cdot \chi(y)$ whenever $x,y\in K$. 

Let  $K$ be a compact 
abelian group and let $\mu$ be its Haar measure. A one-to-one sequence $\{x_n:n\in\N\}$ in $K$ is called {\em uniformly distributed\/} 
provided that
\begin{equation}
\label{eq:intr:1}
\lim_{n\to\infty} \frac{1}{n}\sum_{j=1}^{n}f(x_j)=\int_K f\ d \mu
\ \ 
\mbox{ for every continuous function }
f:K\to \mathbb{C}.
\end{equation}
If $\chi:K\to \mathbb{C}$ is a non-trivial
continuous $\mathbb{S}$-character,
then 
$\displaystyle\int_K \chi\ d \mu=0$, and so the following criterion due to Weyl
says that it suffices to take as $f$ in (\ref{eq:intr:1}) only non-trivial
continuous $\mathbb{S}$-characters:
A one-to-one sequence $\{x_n:n\in\N\}$ in 
$K$ is uniformly distributed if and only if 
$
\lim_{n\to\infty} \frac{1}{n}\sum_{j=1}^{n}\chi(x_j)=0 $
 for every non-trivial continuous $\mathbb{S}$-character $\chi$ of $K$. 
Since the only continuous $\mathbb{S}$-characters of $\T$ are 
of the form 
$x\mapsto e^{2\pi i k x}$ 
for some $k\in\Z$,
for $K=\T$ the Weyl criterion becomes:
A sequence $\{x_n:n\in\N\}$ 
in $\T$ is uniformly distributed if and only if 
$
\lim_{n\to\infty} \frac{1}{n} \sum_{j=1}^n  e^{2\pi i k x_j}=0$ 
for every $k\in \Z$.

For a strictly increasing sequence $S=\{a_n:n\in\N\}$ of integers,  let
$$
\mathrm{Weyl}({S,K})=
\{x\in K: \mbox{ the sequence } \{a_nx:n\in\N\}= Sx \mbox{ is uniformly distributed  in } K\}.
$$
The classical setting of Kronecker-Weyl's theorem  deals with the question of how large is the set $\mathrm{Weyl}({S,K})$. When $K=\T$ (or more generally, when $K=\T^n$ for some 
$n\in\N^+$), the classical result of Weyl says that $\mathrm{Weyl}({S,K})$ is a dense subset of $K$ of (Haar) measure 1. 

\begin{Def}
\label{definition:U:D}
Let $G$ be an abelian group and $K$ a compact abelian group.
\begin{itemize}
\item[(i)]
For a one-to-one sequence $S=\{a_n:n\in\N\}$  in $G$, define 
$$
\mathbb{U}({S,K})=\{h\in \Hom(G,K):h(S)=\{h(a_n):n\in\N\} \mbox{ is uniformly distributed in }K\}.
$$
\item[(ii)] For an infinite subset $S$ of an abelian group $G$, define
$$
\mathbb{D}({S,K})=\{h\in \Hom(G,K):\ h(S) \mbox{ is dense in }K\}.
$$
\end{itemize}
\end{Def}

When $G=\Z$, the map $\theta:\Hom(\Z,K)\to K$ defined by $\theta(h)=h(1)$ for every $h\in \Hom(\Z,K)$, is a  topological isomorphism between $\Hom(\Z,K)$ and $K$ such that  $\theta(\mathbb{U}({S,K}))=\mathrm{Weyl}({S,K})$ for every strictly  increasing sequence $S=\{a_n:n\in\N\}\subseteq \Z$. 
In particular, the sets $\mathbb{U}({S,K})$ and $\mathrm{Weyl}({S,K})$ have the same Borel complexity and are simultaneously dense in $\Hom(\Z,K)$ and $K$, respectively.
This observation allows us to identify the group $K$ with the homomorphism group $\Hom(\Z,K)$ and to focus our attention on the set $\mathbb{U}({S,K})$ instead of the set $\mathrm{Weyl}({S,K})$.

Since uniformly distributed sequences in $K$ are dense,  $\mathbb{U}({S,K})\subseteq \mathbb{D}({S,K})$ for every  one-to-one sequence  $S=\{a_n:n\in\N\}$ in $G$. In fact, a certain converse also holds: If $S$ is a countably infinite subset of $G$ such that $h(S)$ is dense in $K$,  then one can always find a one-to-one enumeration $S=\{a_n:n\in\N\}$ 
of the set $S$ such that the sequence $\{h(a_n):n\in\N\}$ becomes uniformly distributed in $K$.
In other words, {\em it is the density in $K$ that remains from a uniformly distributed sequence  in $K$ after forgetting its enumeration\/}. 
This allows us to consider $\mathbb{D}({S,K})$ as a  natural {\em topological\/} counterpart of the set $\mathbb{U}({S,K})$.

There are other reasons why we prefer the set $\mathbb{D}({S,K})$ to the set  $\mathbb{U}({S,K})$. Indeed, unlike the classical case of  the integers $\Z$, {\em a priori\/} there is no natural order on an arbitrary abelian group $G$ that allows us to index elements of its countably infinite subset $S$. 
The second reason is that the assignment $S \mapsto \mathbb{D}({S,K})$ is monotone (that is, $S\subseteq S'\subseteq G$ implies   $\mathbb{D}({S,K})\subseteq \mathbb{D}({S',K})$), while there is no natural way to make the assignment $S \mapsto \mathbb{U}({S,K})$ monotone.
At last but not least, 
$\mathbb{D}({S,K})$ has nicer descriptive properties than $\mathbb{U}({S,K})$;
 see Proposition \ref{G:delta-sets}.

To keep closer to the classical situation, we shall focus our attention on the case when $K$ is 
a power $\T^\kappa$ of the torus group $\T$.   Observe that $\Hom(G,\T)$ coincides with the Pontryagin dual group $\du{G}$ of $G$, and $\Hom(G,\T^\kappa)$ is naturally isomorphic to the group $\Hom(G,\T)^\kappa\cong \du{G}^\kappa$, for every cardinal $\kappa$. In the future we will always identify $\Hom(G,\T^\kappa)$ with $\du{G}^\kappa$.

Assume that $\tau$ and $\kappa$ are cardinals such that $\tau\le\kappa$, and let $q:\du{G}^\kappa\to\du{G}^\tau$ be the natural projection.
Clearly, $q(\mathbb{D}({S,\T^\kappa}))\subseteq \mathbb{D}({S,\T^\tau})$
for every subset $S$ of an abelian group $G$. In other words, the bigger the cardinal $\kappa$, the ``harder'' it is to send a given set $S$ by a homomorphism $h:G\to \T^\kappa$ into a dense subset $h(S)$ of $\T^\kappa$, and so the ``thinner'' is the set $\mathbb{D}({S,\T^\kappa})$. In fact, there is a natural limit:
If $S$ is countable, then $\mathbb{D}({S,\T^\kappa})\not=\emptyset$ implies that $\T^\kappa$ is separable, which yields $\kappa\le\cont$ by \cite[Theorem 2.3.15]{Eng}.
Therefore, the case $\kappa=\cont$ represents the strongest possible version of any positive result that ensures $\mathbb{D}({S,\T^\kappa})\not=\emptyset$. This explains why $\mathbb{D}({S,\T^\cont})$ appears in items (ii) and (iv) of Theorem \ref{main:corollary}, which is our principal result. This corollary characterizes subsets $S$ of abelian groups $G$ such that $\mathbb{D}({S,\T^\cont})\not=\emptyset$ and, moreover, it demonstrates that a mere non-emptiness of the set $\mathbb{D}({S,\T^\cont})$ automatically guarantees that this set is ``rather big'' in $\Hom(G,\T^\cont)=\du{G}^\cont$. The previous discussion allows us to view Theorem \ref{main:corollary} as an {\em extreme topological version\/} of the Kronecker-Weyl's theorem for arbitrary subsets of abelian groups.

The manuscript is organized as follows. In Section \ref{main:theorem} we collect basic properties of the family $\TT_n(G)$ of \round{n} subsets of an abelian group $G$ ($n\in\N$). All major results in this paper are corollaries of Theorem \ref{main:result} whose proof is postponed until Section \ref{proof:section}.
Section \ref{Kronecker-Weyl:section}
 contains straightforward corollaries of this general theorem. In particular, we show that an abelian group $G$ admits a dense monomorphism in $\T^\cont$ precisely when $|G|\leq 2^\cont$ and $G$ is not bounded  (Corollary \ref{dense:subgroups:T^c}). 
Applications to group actions (Corollary \ref{action}) and
discrete flows (Corollary \ref{flow}) on $\T^\cont$ 
are also given. In Section \ref{Zariski:section} we apply Theorem \ref{main:result} to the problem of existence of precompact group topologies on an abelian group $G$ realizing simultaneously the Zariski closure of countably many subsets of $G$  (Theorem \ref{realizing:Zariski:closure}).
In Section \ref{potentially:dense:section}
we make a significant contribution to a problem of Markov, going back to 1944, asking for an algebraic description of potentially dense sets in groups.
Corollary \ref{potentially:dense}
completely solves this problem for countable subsets of abelian groups,
while Corollary \ref{potentially:dense:subsets}
gives a complete 
description of potentially dense subsets of abelian groups of size at most $2^\cont$.
Even a particular case of this description (given in Corollary \ref{cor:TY}) solves a problem of Tkachenko and Yaschenko \cite[Problem 6.6]{TY}.
In Section \ref{Bohr:section} we apply our principal result to Bohr topologies and Bohr compactifications of abelian groups. In particular, we offer as a corollary an easy direct proof of classical results of Flor (Corollary \ref{Flor:result}) and Glicksberg (Corollary \ref{Glicksberg:result}). 
Sections \ref{dual:group:section} and \ref{Kurama:section} develop tools necessary for the proof of the main result (Theorem \ref{main:result}) that is carried out in  Section \ref{proof:section}. The last Section \ref{section:9} deals with Borel complexity of sets $\mathbb{D}({S,K})$ and $\mathbb{U}({S,K})$.

\section{Main theorem}
\label{main:theorem}

We say that $d\in\mathbb{N}$ is a {\em proper divisor of $n\in\mathbb{N}$\/} provided that $d\not\in\{0,n\}$ and $dm=n$ for some $m\in \mathbb{N}$. 
Note that, according to our definition, each $d\in\N^+$ is a proper divisor of $0$.

\begin{Def}\label{def:of:almost:n:torsion:sets} 
Let $G$ be an abelian group.
\begin{itemize}
\item[(i)] Let $n\in \mathbb{N}$. Following \cite{DS}, we say that  a countably infinite subset $S$ of an abelian group $G$ is {\em \round{n}\/} in $G$ if $S\subseteq G[n]$ and  the set $\{x\in S: dx=g\}$ is finite for  each $g\in G$ and every proper divisor $d$ of $n$.
\item[(ii)] For $n\in \mathbb{N}$, let $\TT_n(G)$ denote the family of all  \round{n} sets in $G$.
\item[(iii)] Define  $\TT(G)=\bigcup\{\TT_n(G):n\in\N\}$.
\end{itemize}
\end{Def}

The notion of an \round{n} set was introduced first in  \cite[Definition 3.3]{DT} under 
a different name and split  into two cases; see \cite[Remark 4.2]{DS} for an extended comparison between this terminology and the one proposed in  \cite{DT}. \Round{n} sets found applications in \cite{DS,DS-MZ,DT,TY}. 

In order to clarify Definition \ref{def:of:almost:n:torsion:sets} and to facilitate future references, we collect basic properties of  \round{n} sets in our next remark.

\begin{remark}\label{Background} 
Let $G$ be an abelian group. 
\begin{itemize}
\item[(i)] $\TT_1(G)=\emptyset$. 
\item[(ii)] $ \TT_n(G)\cap \TT_m(G)=\emptyset$ for distinct $m, n\in\N$.
\item[(iii)] Each family $\TT_n(G)$ is closed
under taking infinite subsets, and
so $\TT(G)$ has the same property.
\item[(iv)] If $H$ is a subgroup of $G$, then $\TT_n(H)= \{S\in\TT_n(G): S\subseteq H\}$ for every $n\in\N$; see \cite[Lemma 4.4]{DS}. In particular,  whether a set $S$ is \round{n} in $G$ depends only on the subgroup of $G$ generated by $S$. 
\item[(v)] If $S\in \TT_n(G)$ for some $n\in\N$,
then the set $\{\pi\in\du{G}:\pi(S)$ is dense in $\T[n]\}$ is dense in $\du{G}$
(see \cite[Lemma 3.7]{DT} for $n\ge 2$ and \cite[Lemma 3.3]{TY} for $n=0$).
In particular, $\mathbb{D}({S,\T})$ is dense in $\du{G}$ for every $S\in \TT_0(G)$.
\item[(vi)] Every infinite subset $X$ of $G$ contains a set of the form $g+S$, where $g\in G$ and $S\in \TT_n(G)$ for some $n\in\N$ \cite[Lemma 3.6]{DT}.
\end{itemize}
\end{remark}

Item (v) of this remark explains why \round{n} sets appear prominently 
in Theorems \ref{main:result} and \ref{main:corollary},
as well as in
Corollary \ref{one:homomorphism}.

\begin{notation}
For an abelian group $G$ and  $E\in \TT(G)$, we use $\n_E$ to denote the unique integer $n\in\N$  such that $E\in \TT_n(G)$. (The uniqueness of such $n$ follows from Remark \ref{Background}(ii).)
\end{notation}

Recall that a space $X$ has the {\em Baire property\/} if 
$\bigcap\{U_n:n\in\N\}$ is dense in $X$ for every sequence 
$\{U_n:n\in\N\}$ of dense open subsets of $X$.

All major results in this paper are corollaries of a single general theorem: 

\begin{thm}
\label{main:result}
For an abelian group $G$ and a countable subfamily $\E\not=\emptyset$ of $\TT(G)$, define
\begin{equation}
\label{Sigma:G}
\mbox{$\Sigma_{G,\E}=\{\sigma\in \Hom(G,\T^\cont):$ 
$\sigma(E)$ is dense in $\T[\n_E]^\cont$ for every $E\in\mathscr{E}\}$.}
\end{equation}
Then:
\begin{itemize}
\item[(i)]
 $\Sigma_{G,\E}$ is a dense subset of $\Hom(G,\T^\cont)=\du{G}^\cont$ having the Baire property;
\item[(ii)] 
$\Sigma_{G,\E}\cap \Mono(G,\T^\cont)$ is a dense subset of $\du{G}^\cont$ having the Baire property if and only if $|G|\le 2^\cont$.
\end{itemize}
\end{thm}

The proof of this theorem is postponed until Section \ref{proof:section}.

\begin{remark}
The statement of 
Theorem \ref{main:result} appears to be the best possible.
\begin{itemize}
\item[(i)]
One cannot strengthen the conclusion of Theorem \ref{main:result} by replacing $\cont$ in it with a cardinal $\kappa>\cont$. Indeed, if $E\in \E$, then  $\T[\n_E]^\kappa$ must be separable, 
and since $\n_E\not=1$ by Remark \ref{Background}(i), this implies $\kappa\le \cont$; see, for example, \cite[Theorem 2.3.15]{Eng}.
\item[(ii)]
One cannot  strengthen the statement of Theorem \ref{main:result} by increasing the size of the family $\E$;  see Remark \ref{restriction:on:size}(ii). 
\item[(iii)] 
 One cannot strengthen  the Baire property of $\Sigma_{G,\E}$ to requiring $\Sigma_{G,\E}$ to be a dense $G_\delta$-subset of $\du{G}^\cont$. In fact, $\Sigma_{G,\E}$ does not even contain any non-empty $G_\delta$-subset of $\du{G}^\cont$; 
see Remark \ref{non:G-delta}(ii). 
\end{itemize}
\end{remark}

\begin{remark}
\label{Baire:remark}
Let $X$ be a dense subspace of a space $Y$. 
\begin{itemize}
\item[(i)]
$X$ has the Baire property if and only if $X\cap W\cap\bigcap\mathscr{U}\not=\emptyset$
whenever $W$ is a non-empty open subset of $X$ and $\mathscr{U}$ is a countable family of open dense subsets of $Y$.
\item[(ii)]
If $X$ is a dense subspace of $Y$ having the Baire property, then so is every space $Z$ satisfying $X\subseteq Z\subseteq Y$.
\end{itemize}
\end{remark}

\section{A Kronecker-Weyl theorem for subsets of abelian groups}
\label{Kronecker-Weyl:section}

We start by providing a convenient reformulation of Definition \ref{def:of:almost:n:torsion:sets} in the case $n=0$:  A subset $E$ of an abelian group $G$ is  \round{0} if and only if $E$ is countably infinite, but $E\cap (g+G[k])$ is finite whenever $g\in G$ and $k\in \N^+$.

The next theorem is the principal result of this paper. 
\begin{thm}
\label{main:corollary}
For a subset $S$ of an abelian group $G$, the following conditions are equivalent:
\begin{itemize}
\item[(i)]
$\mathbb{D}({S,\T^\kappa})\not=\emptyset$ for some cardinal $\kappa\ge 1$,
\item[(ii)]
$\mathbb{D}({S,\T^\cont})$ is a dense subset of $\du{G}^\cont$ having the Baire property,
\item[(iii)] $S$ contains an \round{0} set.
\end{itemize}
Furthermore, 
if one additionally assumes that $|G|\le 2^\cont$, then the following 
item can be added to the list of equivalent conditions (i)--(iii):
\begin{itemize}
\item[(iv)] 
$\mathbb{D}({S,\T^\cont})\cap\Mono(G,\T^\cont)$ 
is a dense subset of $\du{G}^\cont$ having the Baire property.
\end{itemize}
\end{thm}

\begin{proof} The implication (ii)$\to$(i) is clear.

(i)$\to$(iii) Let $\kappa$ be a cardinal from (i). Fix $\sigma\in \mathbb{D}({S,\T^\kappa})$. Since $\sigma(S)$ is dense in $\T^\kappa$, for every $n\in\N$ the set
$n \sigma(S)$ must be dense in $n\T^\kappa=\T^\kappa$. In particular,
$nS$ must be infinite for every $n\in\N^+$. Then $S$ contains an 
\round{0} subset by \cite[Proposition 5.11]{DS-MZ}.

The implication (iv)$\to$(ii) follows Remark \ref{Baire:remark}(ii).

(iii)$\to$(ii)
Let $E\subseteq S$ be an \round{0} set. Define $\E=\{E\}$.  Since $\T[0]=\T$, \eqref{Sigma:G} yields $\Sigma_{G,\E}\subseteq \mathbb{D}({S,\T^\cont})$. From this and Theorem \ref{main:result}(i), we get (ii).

(iii)$\to$(iv)  Assume now that $|G|\le 2^\cont$. Let $E$ and $\mathscr{E}$ be as in the proof of the implication  (iii)$\to$(ii). Since  $\Sigma_{G,\E}\cap \Mono(G,\T^\cont)\subseteq \mathbb{D}({S,\T^\cont})\cap \Mono(G,\T^\cont)$, and the former set is a dense subset of $\du{G}^\cont$ having the Baire property by Theorem \ref{main:result}(ii), so is the latter set; see Remark \ref{Baire:remark}(ii).
\end{proof}

One cannot strengthen items (ii) or (iv) of Theorem \ref{main:corollary}
by requiring $\mathbb{D}({S,\T^\cont})$ to be a dense $G_\delta$-subset of $\du{G}^\cont$. Indeed, we shall show in Remark \ref{non:G-delta} that $\mathbb{D}({S,\T^\cont})$ does not even contain any 
non-empty $G_\delta$-subset of $\du{G}^\cont$.

\begin{cor}
\label{one:homomorphism}
For a subset $S$ of an abelian group $G$, the following conditions are equivalent:
\begin{itemize}
\item[(i)] there exists a homomorphism $\pi:G\to \T^\cont$ such that 
$\pi(S)$ is dense in $\T^\cont$,
\item[(ii)] $S$ contains an \round{0} set.
\end{itemize}
Furthermore, if one additionally assumes that $|G|\le 2^\cont$, then the following  item can be added to the list of equivalent conditions (i) and (ii):
\begin{itemize}
\item[(iii)] there exists a monomorphism $\pi:G\to \T^\cont$ such that  $\pi(S)$ is dense in $\T^\cont$.
\end{itemize}
\end{cor}

The version of this corollary for homomorphisms (monomorphisms) into $\T^\kappa$ for cardinals $\kappa> \cont$ is proved in our paper \cite{DS_HMP}.

According to the classical Kronecker theorem, for every $n$-tuple $(\alpha_1, \ldots, \alpha_n)$ of real numbers the cyclic subgroup of $\T^n$ generated by $(\alpha_1, \ldots, \alpha_n)$ is dense
in $\T^n$ if and only if $1,\alpha_1, \ldots, \alpha_n$ are rationally independent. 
This implies the well-known fact that $\T^\cont$ contains a dense cyclic subgroup $\Z$.
In our next corollary we describe the  abelian groups that admit a dense embedding into $\T^\cont$:

\begin{cor} \label{dense:subgroups:T^c}
An abelian group $G$  is isomorphic to a dense subgroup of  $\T^\cont$ if and only if $G$ is not bounded and $|G|\leq 2^\cont$.
\end{cor}

\begin{proof}
Suppose that an abelian group $G$ is not bounded and  $|G|\leq 2^\cont$. By \cite[Corollary 5.12]{DS-MZ}, $G$ contains an \round{0} set $S$.
Applying Corollary \ref{one:homomorphism}, we can find a monomorphism  $\pi:G\to \T^\cont$ such that $\pi(S)$ is dense in $\T^\cont$. Then
$\pi(G)$ is dense in $\T^\cont$ as well. The reverse implication is clear.
\end{proof}

Every homomorphism $\pi:G\to \T^\cont$ of an abelian group $G$ defines the action  $(g,x)\mapsto gx$ of $G$ on $\T^\cont$ by $gx=\pi(g)+x$ for $g\in G$ and $x\in \T^\cont$. In this language Corollary \ref{one:homomorphism} can be restated as follows.

\begin{cor}
\label{action}
Let $S$ be a subset of an abelian group $G$ such that $S$ contains an \round{0} set. Then there exists an action of $G$ on $\T^\cont$ by homeomorphisms of $\T^\cont$ such that the ``$S$-orbit'' $\{gx:g\in S\}$ of each point $x\in \T^\cont$ is dense in $\T^\cont$.
\end{cor}

The following particular case of Corollary \ref{action} seems to be  new as well.

\begin{cor}
\label{flow}
For every infinite subset $S$ of $\Z$, there exists a translation $f:\T^\cont\to\T^\cont$ of the group
$\T^\cont$ such that the ``$S$-orbit'' $\{f^n(x):n\in S\}$ of each point $x\in \T^\cont$ is dense in $\T^\cont$. 
\end{cor}

\begin{proof}
Follows from Corollary \ref{action} and the fact that every infinite subset of $\Z$ is \round{0}.
\end{proof}

The the best of our knowledge, the following application gives a new contribution to diophantine approximation:  

\begin{cor} 
\label{diophantine}
For every infinite set $S$ of integers, there exists an indexed set $\{x_\alpha:\alpha<\cont\}\subseteq [0,1)$ (depending on $S$ and) having the following property:  If  $\varepsilon>0$, 
$k\in\N^+$, 
$y_1,\dots,y_k\in\mathbb{R}$
and 
$\alpha_1<\cont, \ldots, \alpha_k<\cont$, then one can find $s\in S$ and  $n_1, \ldots ,n_k\in\Z$ such that 
$|sx_{\alpha_j}-y_j-n_j|< \varepsilon$ for every $j=1,2,\ldots, k$.  
\end{cor}

\begin{proof}
Apply Corollary \ref{one:homomorphism} to $G=\Z$ to find a monomorphism  $\pi:\Z\to \T^\cont$ such that 
$\pi(S)$ is dense in $\T^\cont$. Let $\pi(1)=\{t_\alpha\}_{\alpha<\cont}\in \T^\cont$.
For every $\alpha<\cont$, choose $x_\alpha\in[0,1)$ such that $\psi(x_\alpha)=t_\alpha$, where $\psi:\mathbb{R}\to \mathbb{R}/\mathbb{Z}=\T$ is the natural quotient homomorphism.  Then $\{x_\alpha:\alpha<\cont\}$ has the desired properties.
\end{proof}

It should be noted that a much weaker version of Corollary \ref{diophantine}, with $\cont$ replaced by $\omega$, follows from results of \cite{TY}.

\section{Realization problem for the Zariski closure}
\label{Zariski:section}

Let $G$ be an abelian group. According to Markov \cite{Mar}, a set of the form $g+G[n]$, for a suitable $g\in G$ and  $n\in\N$, is called an {\em elementary algebraic\/} subset of $G$, and arbitrary intersections of finite unions of elementary algebraic subsets of $G$ are called {\em algebraic\/} subsets of $G$. One can easily see that the family of all algebraic subsets of $G$ is closed under finite unions and arbitrary intersections, and  contains $G$ and all finite subsets of
$G$; thus,   it  can be taken as the family of closed sets of a unique $T_1$ topology ${\mathfrak Z}_G$ on $G$. Markov \cite{Mar1, Mar} defined  the {\em algebraic closure\/} of a subset $X$ of $G$ as the intersection of all algebraic subsets of $G$ containing $X$, i.e., the smallest algebraic set that contains $X$. This definition satisfies the conditions necessary for introducing
a topological closure operator on $G$. Since a topology on a set is uniquely determined by its closure operator, it is fair to say that Markov was the first to (implicitly) define   the topology ${\mathfrak Z}_G$,  though he did not  name it. To the best of our knowledge, the first name for this topology appeared explicitly in print in a 1977 paper by Bryant \cite{Bryant},  who called it a {\em verbal topology\/} of $G$.  In a more recent  series of papers 
beginning with \cite{BMR}, Baumslag, Myasnikov and Remeslennikov have developed algebraic geometry over an abstract group $G$.  In an analogy with the Zariski topology from algebraic geometry, they introduced the  Zariski topology on the finite powers $G^n$ of a group $G$. In the particular case  when $n=1$, this topology coincides with the verbal topology of Bryant. For this reason, the topology ${\mathfrak Z}_G$ is also called the {\em Zariski topology\/} of $G$ in \cite{DS_OPIT, DS_JGT}.  
The topology ${\mathfrak Z}_G$ is Noetherian, and so compact. While ${\mathfrak Z}_G$ is a $T_1$ topology, it is Hausdorff only when $G$ is finite \cite{DS-MZ}. A comprehensive study of the Zariski topology is carried out in \cite{DS-MZ}.

Given a topology $\mathcal{T}$ on $G$, we  denote  by $\CL_{\mathcal{T}}(X)$ the $\mathcal{T}$-closure of a set $X\subseteq G$. We 
call $\CL_{\Zar_G}(X)$ the {\em Zariski closure\/} of  $X$ in $G$. (Thus, $\CL_{\Zar_G}(X)$ is the algebraic closure of  $X$
in the terminology of Markov \cite{Mar1, Mar}). For every Hausdorff group topology $\mathcal{T}$ on $G$,  one has
  $\Zar_G\subseteq  \mathcal{T}$, and therefore, $\CL_{\mathcal{T}}(X)\subseteq \CL_{\Zar_G}(X)$ for each set $X\subseteq G$. This  inclusion naturally leads to the following {\em realization problem for the Zariski closure\/}:
For a given set $X\subseteq G$, can one always find a Hausdorff group topology $\mathcal{T}$ on $G$ 
 such that $\CL_{\mathcal{T}}(X)= \CL_{\Zar_G}(X)$?  This problem was first considered  by Markov in \cite{Mar}, who proved that for every subset $X$ of a countable group  $G$,
 there exists a metric group topology $\mathcal{T}$ on $G$ such that $\CL_{\mathcal{T}}(X)= \CL_{\Zar_G}(X)$.  We make the following contribution to this general problem  in the abelian case:

\begin{thm}
\label{realizing:Zariski:closure}
Let $G$ be an abelian group of size at most $2^\cont$, and let $\mathscr{X}$ be a countable family of subsets of $G$. Then there exists a precompact Hausdorff group topology
$\mathcal{T}$ on $G$ such that the $\mathcal{T}$-closure of each $X\in \mathscr{X}$ coincides with its 
Zariski
closure.
\end{thm}

\begin{proof} According to \cite[Theorem 7.1]{DS-MZ}, for every $X\in\mathscr{X}$ 
there exist a finite family  $\E_X\subseteq\TT(G)$ and finite sets $F_X\subseteq G$ 
and $\{h_{E,X}:E\in\E_X\}\subseteq G$ such that 
\begin{equation}
\label{eq:5}
F_X\cup \bigcup_{E\in\E_X} h_{E,X}+E\subseteq X
\ \ 
\mbox{ and }
\ \ 
\CL_{\Zar_G}(X)=F_X\cup\bigcup_{E\in\E_X} h_{E,X}+G[\n_E].
\end{equation}
Applying Theorem \ref{main:result}(ii) to $G$ and  $\E=\bigcup\{\E_X: X\in\mathscr{X}\}$,
we can  find a monomorphism $\sigma: G\to \T^\cont$ 
such that $\sigma\in\Sigma_{G,\mathscr{E}}$. Without loss of generality, we shall identify $G$ with the subgroup $\sigma(G)$ of $\T^\cont$. That is, we shall assume that $G\subseteq \T^\cont$ and $\sigma$ is the identity map. Under these assumptions, from \eqref{Sigma:G} we conclude that
each $E\in\E$ is dense in $\T[\n_E]^\cont$. We claim that the precompact group topology $\mathcal{T}$ induced on $G$ by the topology of $\T^\cont$ has the desired property. 

Indeed, let $X\in \mathscr{X}$.  Fix $E\in \E_X$. Since $\E_X\subseteq\E$, $E$ is dense in $\T[\n_E]^\cont$. Since $E\subseteq G[\n_E]=G\cap \T[\n_E]^\cont$, we conclude that  $E$ is $\mathcal{T}$-dense in  the $\mathcal{T}$-closed set $G[\n_E]$. Thus,
$\CL_{\mathcal{T}}(E)=G[\n_E]$. We showed that $\CL_{\mathcal{T}}(E)=G[\n_E]$ for every $E\in \E_X$. From this and (\ref{eq:5}), we obtain
$$
\CL_{\Zar_G}(X)= F_X\cup\bigcup_{E\in\E_X} h_{E,X}+G[\n_E]= F_X\cup\bigcup_{E\in\E_X} h_{E,X}+\CL_{\mathcal{T}} (E)=
$$
$$
F_X\cup\bigcup_{E\in\E_X} \CL_{\mathcal{T}} (h_{E,X}+E)=
\CL_{\mathcal{T}}\left(F_X\cup\bigcup_{E\in\E_X} h_{E,X}+E\right)\subseteq \CL_{\mathcal{T}}(X).
$$
The reverse inclusion $\CL_{\mathcal{T}}(X)\subseteq \CL_{\Zar_G}(X)$ follows from the inclusion $\Zar_G\subseteq \mathcal{T}$.
\end{proof}

Our next remark shows that one cannot increase the size of the family $\mathscr{X}$ in Theorem \ref{realizing:Zariski:closure}.

\begin{remark}
\label{restriction:on:size}
\begin{itemize}
\item[(i)]
Theorem \ref{realizing:Zariski:closure} fails for $G=\Z$ and the family of all subsets of $G$ taken as $\mathscr{X}$. Indeed, if our theorem were true in such a case, then the Zariski topology on $G$ would coincide with $\mathcal{T}$. Since the Zariski topology for $\Z$ is co-finite (and thus, non-Hausdorff),
we get  a contradiction. 
\item[(ii)]
Theorem \ref{main:result} fails for $G=\Z$ and the family $\mathscr{E}$ of all infinite subsets of $G$.
(Note that $\mathscr{E}= \TT_0(G)\subseteq \TT(G)$.)
Indeed, a careful analysis of the proof of Theorem \ref{realizing:Zariski:closure}
shows that, if Theorem \ref{main:result} would hold for such $G$ and $\E$, 
then Theorem \ref{realizing:Zariski:closure} would also hold for $G$ and $\mathscr{X}$ from item (i), giving a contradiction.
\end{itemize}
\end{remark}

\begin{cor} 
\label{corollary:for:realizing:closure}
For an abelian group $G$, the following conditions are equivalent:
\begin{itemize}
  \item[(i)] $|G|\le 2^\cont$,
  \item[(ii)] for every subset $X$ of $G$, there exists a precompact Hausdorff group topology $\mathcal{T}_{{X}}$ on $G$ such that $\CL_{\mathcal{T}_{{X}}}(X)=\CL_{\Zar_G}(X)$,
  \item[(iii)] for every countable family $\mathscr{X}$ of subsets of $G$, one can find a precompact Hausdorff group topology $\mathcal{T}_{\mathscr{X}}$ on $G$ such that $\CL_{\mathcal{T}_{\mathscr{X}}}(X)=\CL_{\Zar_G}(X)$ for every $X\in\mathscr{X}$.
\end{itemize}
\end{cor}
\begin{proof}
The implication (i)$\to$(iii) is proved in Theorem \ref{realizing:Zariski:closure}. The implication (iii)$\to$(ii) is trivial. Let us prove the implication (ii)$\to$(i). 
According to \cite[Corollary 8.9]{DS-MZ}, the topology $\Zar_G$ is separable, so there exists a countable set $X\subseteq G$ with $\CL_{\Zar_G}(X)=G$.
Applying item (ii) to this $X$, we can choose a (precompact) Hausdorff group topology $\mathcal{T}_{{X}}$ on $G$ such that $\CL_{\mathcal{T}_{{X}}}(G)=\CL_{\Zar_G}(X)=G$.
That is, $\mathcal{T}_{{X}}$  is separable.  Since $\mathcal{T}_{{X}}$ is Hausdorff, this yields (i) by 
\cite{Pospishil}; see also  \cite[Theorem 1.5.3]{Eng}.
\end{proof}

A counterpart of this corollary, with the word "metric" added to both items (ii) and (iii),  and the inequality in item (i) strengthened to $|G|\le \cont$, is proved in our paper \cite{DS-MZ}.

\section{Markov's potential density}
\label{potentially:dense:section}

According to Markov \cite{Mar}, a subset $X$ of a group $G$ is {\em potentially dense\/} in $G$ if $G$ admits a Hausdorff  group topology $\mathcal{T}$ such that $X$ is $\mathcal{T}$-dense in $G$. The last section of  Markov's paper \cite{Mar} is exclusively dedicated to the following problem: {\em which subsets of a group $G$ are potentially dense in $G$?\/} Markov showed that
every infinite subset of $\Z$ is potentially dense in $\Z$ \cite{Mar}. This was strengthened in \cite[Lemma 5.2]{DT0} by proving that every infinite subset of $\Z$ is dense in some precompact metric group topology on $\Z$. 
(The authors of \cite{Mar} and \cite{DT0} were  apparently
unaware that both these results easily follow from  the uniform distribution theorem of Weyl \cite{W}.) Further progress was  made by Tkachenko and Yaschenko \cite{TY},
 who proved the following theorem: If an abelian group $G$ of size at most $\cont$  is either almost  torsion-free or has exponent $p$ for  some prime number $p$, then every infinite subset of $G$ is potentially dense in $G$. (According to \cite{TY}, an abelian group $G$ 
is {\em almost torsion-free\/} if the $p$-rank of $G$ is finite for every prime  number $p$.)  

In \cite{DS_HMP}, the authors resolved  Markov's potential density problem for uncountable subsets of divisible and almost torsion-free abelian groups, among other classes. Our
next two corollaries provide a solution to Markov's problem for arbitrary subsets of abelian groups of size at most $2^\cont$ and for countable subsets of arbitrary abelian groups, respectively.

\begin{cor} 
\label{potentially:dense:subsets}
Let $X$ be a subset of an abelian group $G$ such that $|G|\le 2^\cont$. Then the following conditions are equivalent:
\begin{itemize}
 \item[(i)] $X$ is potentially dense in $G$,
 \item[(ii)] $X$ is $\mathcal{T}$-dense in $G$ for some precompact Hausdorff group topology  $\mathcal{T}$ on $G$,
 \item[(iii)] $\CL_{\Zar_G}(X)=G$.
\end{itemize}
\end{cor}

\begin{proof} The implication (iii)$\to$(ii)  follows from  Corollary \ref{corollary:for:realizing:closure}. Clearly, (ii) implies (i).  To prove the implication (i)$\to$(iii), 
let $\mathcal{T}$ be a Hausdorff group topology on $G$ such that  $X$ is $\mathcal{T}$-dense in $G$. Then $G=\CL_{\mathcal{T}}(X) \subseteq \CL_{\Zar_G}(X)\subseteq G$, because $\Zar_G\subseteq \mathcal{T}$. Therefore, $\CL_{\Zar_G}(X)=G$.
\end{proof}

\begin{cor}
\label{potentially:dense}
For a countably infinite subset $X$ of an abelian group $G$, the following conditions are equivalent:
\begin{itemize}
\item[(i)]  $X$ is potentially dense in $G$,
\item[(ii)]  there exists a precompact Hausdorff group topology  $\mathcal{T}$ on $G$ such that $X$ is $\mathcal{T}$-dense in $G$,
\item[(iii)]  $|G|\le 2^\cont$ and $\CL_{\Zar_G}(X)=G$.
\end{itemize}
\end{cor}

\begin{proof} 
Assume (i). Then $X$ is $\mathcal{T}$-dense in $G$ for some Hausdorff group topology $\mathcal{T}$ on $G$, and so $\mathcal{T}$ is separable, which yields $|G|\le 2^\cont$ by \cite{Pospishil}; see also 
\cite[Theorem 1.5.3]{Eng}. The rest follows from Corollary \ref{potentially:dense:subsets}.
\end{proof}

A counterpart of Corollaries \ref{potentially:dense:subsets} and \ref{potentially:dense}, with the word ``metric'' added to their items (ii), and  the condition on a group $G$ strengthened to $|G|\le \cont$, is proved in our  paper \cite{DS-MZ}. 

We say that a subset $X$ of an abelian group $G$ is {\em Zariski dense\/} in $G$
provided that $\CL_{\Zar_G}(X)=G$. Items (iii) of Corollaries \ref{potentially:dense:subsets} and \ref{potentially:dense} make it important to characterize Zariski dense sets. One such ``characterization'' simply follows from the definition.
A subset $X$ of an abelian group $G$ is Zariski dense in $G$ provided that, if $k\in \N$, $g_1, g_2, \ldots, g_k\in G$, $n_1, n_2, \ldots, n_k\in \N$ and  each $x\in X$ satisfies the equation $n_ix=g_i$ for some $i= 1,2,\ldots, k$ (depending on $x$), then every $g\in G$ also satisfies some equation $n_jg=g_j$, for a suitable $j= 1,2,\ldots, k$.  A complete description in terms of \round{n} sets is given in \cite{DS-MZ}. If an abelian group $G$ is not bounded, then a subset $X$ of $G$ is Zariski dense in $G$ if and only if  $nX$ is infinite for every $n\in\N^+$ if and only if $X$ contains an \round{0} set \cite[Theorem 7.4]{DS-MZ}. A subset $X$ of an infinite bounded abelian group $G$ is Zariski dense in $G$ if and only if $g+X$ contains an \round{m} set for every $g\in G$, where $m$ is the smallest positive integer such that $mG$ is finite \cite[Theorem 7.1]{DS-MZ}. 

As was mentioned in the text preceding Corollary \ref{potentially:dense:subsets}, Tkachenko and Yaschenko proved  in \cite{TY}  that if $|G|\leq \cont$ and  $G$ is either almost torsion-free or has exponent $p$ for some prime number $p$, then every infinite subset of $G$ is potentially dense in $G$.  In the same manuscript, the authors asked  whether the restriction $|G|\leq \cont$ in their result can be weakened to $|G|\le 2^\cont$ (\cite[Problem 6.6]{TY}). As a special case of the above corollary, we obtain a positive solution to this problem. 

\begin{cor} 
\label{cor:TY}
For an abelian group $G$ with $|G|\le 2^\cont$, the following conditions are equivalent: 
\begin{itemize}
\item[(i)]   every infinite subset of $G$ is potentially dense in $G$,
\item[(ii)]   every infinite subset of $G$ is $\mathcal{T}$-dense in $G$ for some precompact Hausdorff group topology  $\mathcal{T}$ on $G$,
\item[(iii)]   $G$ is either almost  torsion-free or has exponent $p$ for some prime number $p$,
\item[(iv)]   every $\Zar_G$-closed subset of $G$ is finite. 
\end{itemize}
\end{cor}
\begin{proof}
The equivalence of (iii) and (iv) is clear from the definition of the Zariski topology (and was observed already in \cite{TY}). Notice that (iv) is equivalent to having $\CL_{\Zar_G}(X)=G$ for every infinite subset $X$ of $G$. The rest follows from Corollary \ref{potentially:dense}.
\end{proof}

Note that if some countably infinite set is potentially dense in $G$, then $|G|\le 2^\cont$  by the implication (i)$\to$(iii) of Corollary \ref{potentially:dense}, 
so this cardinality restriction in Corollary \ref{cor:TY} is necessary.

\section{Applications to Bohr topologies and Bohr compactifications}
\label{Bohr:section}

Let $G$ be an abelian group. The strongest precompact group topology on $G$ is called the  {\em Bohr topology\/} of $G$, and we use  $G^\#$ to denote  
the group $G$ equipped with its Bohr topology. The completion $bG$ of $G^\#$ is a compact abelian group
called the {\em Bohr compactification\/} of $G$.  The terms  Bohr topology and Bohr compactification have been chosen 
as a reward to Harald Bohr for his work \cite{B} on almost periodic functions closely related to the Bohr compactification. 

In this section, we illustrate the power of Theorem \ref{main:result} by offering simple proofs of some well-known properties of Bohr topologies and Bohr compactifications of abelian groups. We start with a result of van Douwen; see \cite[Theorem 1.1.3(a)]{vD}.

\begin{cor}
\label{closures:in:the:Bohr:compactification}
Let $G$ be an abelian group and $X$ its infinite subset.
Then the closure $\overline{X}$ of $X$ in $bG$ has size at least $2^\cont$.
\end{cor}

\begin{proof} Let $g$ and $S$ be as in Remark \ref{Background}(vi). Then $\n_S\not=1$ by Remark \ref{Background}(i).  Apply
Theorem \ref{main:result}(i) with $\mathscr{E}=\{S\}$
to fix a homomorphism $\sigma:G\to \T^\cont$ such that $\sigma(S)$ is dense in $\T[\n_S]^\cont$. 
Let $\pi: bG\to \T^\cont$ be the homomorphism extending $\sigma$. Note that $g+\overline{S}=\overline{g+S}\subseteq \overline{X}$, and hence $\pi(g)+\pi(\overline{S})=\pi(g+\overline{S})\subseteq \pi(\overline{X})$. Since $\overline{S}$ is a closed subset of the compact space $bG$, it is compact, and so 
is $\pi(\overline{S})$. In particular, $\pi(\overline{S})$ is closed in $\T^\cont$. Since $\sigma(S)=\pi(S)\subseteq \pi(\overline{S})$, $\sigma(S)$ is dense in $\T[\n_S]^\cont$ and $\pi(\overline{S})$ is closed in $\T^\cont$, we conclude that $\pi(\overline{S})=\T[\n_S]^\cont$. In particular, 
$|\overline{X}|\ge |g+\overline{S}|=|\overline{S}|\ge |\pi(\overline{S})|= |\T[\n_S]^\cont|= 2^\cont$.
\end{proof}

The above corollary immediately yields the well-known result of  Flor \cite{Flor}:

\begin{cor}
\label{Flor:result}
If $G$ is an abelian group and $bG$ is its Bohr compactification, then no sequence of points in $G$ converges in $bG$.
\end{cor}

The well-known theorem of Glicksberg \cite{Glicksberg} also becomes an easy corollary:

\begin{cor}
\label{Glicksberg:result}
An abelian group $G^\#$ with the Bohr topology has no infinite compact subsets.
\end{cor}

\begin{proof} Assume that $K$ is in infinite compact subset of $G^\#$. Choose a countable set $X\subseteq K$, and let $H$ be the smallest subgroup of $G$ containing $X$.
Then $H$ is closed in $G^\#$ \cite[Lemma 2.1]{CSa}, and so $H\cap K$ is a closed subset of $K$. Therefore, $H\cap K$ is a countable compact subset of $G^\#$, and thus, of 
$bG$. Since $X\subseteq H\cap K$, it follows that $\overline{X}\subseteq \overline{H\cap K}= H\cap K$ and $|\overline{X}|\le |H\cap K|\le\omega$. Since $X$ is infinite, 
this contradicts Corollary \ref{closures:in:the:Bohr:compactification}.
\end{proof}

Recall that a space $X$ is {\em pseudocompact\/} if every real-valued continuous function defined on $X$ is bounded \cite{Hewitt}.
The following generalization of Glicksberg's theorem, due to Comfort  and 
Trigos-Arrieta \cite{Trigos}, can also  be easily derived:
\begin{cor}
An abelian group $G^\#$ with the Bohr topology has no infinite pseudocompact subsets.
\end{cor}

\begin{proof} Assume that $X$ is an  infinite pseudocompact subspace of $G^\#$. Take a countably infinite  subset $S$ of $X$. Let $H$ be the divisible hull of $G$, and let $H_1$ be the smallest divisible subgroup of $H$ containing $S$. Then $H_1$ is countable. Since $H_1$ is divisible, $H_1$ splits, i.e. there exists an abelian group $N$ such that $H=H_1\oplus N$; see \cite{Fu}. Let $\varphi:H\to H_1$ be the natural projection. Since $\varphi:H^\#\to H_1^\#$ is continuous, $\varphi(X)$ is a pseudocompact subset of $H_1^\#$. Since $|\varphi(X)|\le |H_1|\le\omega$, $\varphi(X)$ is compact by \cite[Theorems 3.11.1 and 3.11.12]{Eng}. Thus, $\varphi(X)$ is finite by 
Corollary \ref{Glicksberg:result}.  Since $S=\varphi(S)\subseteq \varphi(X)$, the set $S$ must be finite as well, giving a contradiction.
\end{proof}

\section{General preliminaries}
\label{dual:group:section}

\begin{lemma}
\label{openess:of:sets}
If $S$ is a subset of an abelian group $H$ and $O$ is an open subset of an abelian topological group $K$, then the set 
$P_{S,O}=\{\pi\in \Hom(H,K):\pi(S) \cap O \ne \emptyset \}$ is open in $\Hom(H,K)$.
\end{lemma}
\begin{proof}
Assume that $\pi_0\in P_{S,O}$. Pick  $x\in  \pi_0^{-1}(O) \cap S$. Then $W=\{\pi\in \Hom(H,K): \pi(x)\in O\}$ is an open subset of 
$\Hom(H,K)$ such that $\pi_0\in W\subseteq P_{S,O}$.
\end{proof}

\begin{lemma}
\label{monomorphisms:are:dense:G-delta:for:countable:groups}
If $H$ is a countable abelian group and $Y$ is a countably infinite set, then  $\Mono(H,\allowbreak\T^Y)$ is a dense $G_\delta$-subset of $\du{H}^Y$.
\end{lemma}

\begin{proof} For every $h\in H$, the set $V_h=\{\pi\in \du{H}^Y: \pi(h)\not=0\}$ is open in  $\du{H}^Y$ by Lemma \ref{openess:of:sets}, and so
$\Mono(H,\T^Y)=\bigcap\{V_h: h\in H\setminus\{0\}\}$ is a $G_\delta$-set in $\du{H}^Y$. 

Let $O$ be  an arbitrary non-empty open  subset of $\du{H}^Y$. There exists a non-empty open subset $V$ of $\du{H}^F$ such that 
$V\times \du{H}^{Y\setminus F}\subseteq O$. Choose $\rho\in V$. Since $|H|=|Y\setminus F|=\omega$,
there exists a monomorphism $\xi: H\to \T^{Y\setminus F}$.  Then $\sigma=(\rho,\xi)\in V\times \du{H}^{Y\setminus F}\subseteq O$
 is a monomorphism as well. This proves that $\sigma\in O\cap \Mono(H,\T^Y)\not=\emptyset$. Therefore, $\Mono(H,\T^Y)$ is dense in $\du{H}^Y$. 
\end{proof}

Given a subset $Y$ of a set $X$ and a subgroup $H$ of an abelian group $G$, we define the map $q^{XY}_{GH}:\du{G}^X\to \du{H}^Y$  by  $q^{XY}_{GH}\left(\{\chi_x\}_{x\in X}\right)=\{\chi_x\restr{H}\}_{x\in Y}$ for every  $\{\chi_x\}_{x\in X}\in \du{G}^X$. Note that, after the natural identification, $\du{G}^X=\Hom(G,\T^X)$ becomes a closed subgroup of $\T^{X\times G}$, and the map $q^{XY}_{GH}$ becomes the restriction  to $\du{G}^X$ of the natural projection map $p^{XY}_{GH}: \T^{X\times G}\to \T^{Y\times H}$. In particular, the map $q^{XY}_{GH}$ is continuous.

\begin{lemma}
\label{reflection:lemma}
Assume that  $Z$ is a countable subset of an abelian group $G$,  $X$ is an infinite set and $\mathscr{V}$ is a countable family of open  subsets of $\du{G}^X$.
Then there exist a countable subgroup $H$ of $G$ containing $Z$, a countably infinite set $Y\subseteq X$ and a family $\mathscr{U}=\{U_V:V\in\mathscr{V}\}$ of  open subsets of $\du{H}^Y$  such that $\left(q^{XY}_{GH}\right)^{-1}(U_V)$ is a dense subset of $V$
for every $V\in\mathscr{V}$.
\end{lemma}

\begin{proof}
Fix $V\in\mathscr{V}$. Let $\mathscr{B}_V$ be the family of basic open subsets $B$ of the product $\T^{X\times G}$ such that $\emptyset\not=B\cap \du{G}^X\subseteq V$.
We apply Zorn's lemma to select a maximal subfamily $\mathscr{W}_V$ of $\mathscr{B}_V$ consisting of pairwise disjoint subsets of $V$. Since $\du{G}^X$ is a compact group and every non-empty open subset of $\du{G}^X$ has a positive Haar measure, $\mathscr{W}_V$ must be countable. 
Therefore, $\mathscr{W}=\bigcup_{V\in \mathscr{V}} \mathscr{W}_V$ is a countable family of basic open subsets of $\T^{X\times G}$, so there exist countable sets $I\subseteq G$ and $J\subseteq X$ such that each element of $\mathscr{W}$ depends only on coordinates in $I\times J$; that is, 
\begin{equation}
\label{elements:of:W:are:saturated}
W=\left(p^{XJ}_{GI}\right)^{-1}\left(p^{XJ}_{GI}(W)\right)
\mbox{ for every }
W\in\mathscr{W}.
\end{equation}

Fix a countably infinite subset $Y_0$ of $X$. Then $Y=J\cup Y_0$ is 
countably infinite, and
the subgroup $H$ 
 of $G$ generated by $I\cup Z$ is countable as well.
For typographical reasons, we let $p=p^{XY}_{GH}$ and $q=q^{XY}_{GH}$.

Let $V\in\mathscr{V}$ be arbitrary. The set $V'=\left(\bigcup\mathscr{W}_V\right)\cap \du{G}^X$ is dense in $V$ by maximality of $\mathscr{W}_V$. Since $I\subseteq H$ and $J\subseteq Y$, from \eqref{elements:of:W:are:saturated} it follows that  $W=p^{-1}(p(W))$ for every  $W\in\mathscr{W}$. Since $\mathscr{W}_V\subseteq \mathscr{W}$, this gives
$\bigcup\mathscr{W}_V=p^{-1}(p(\bigcup\mathscr{W}_V))$, and so $V'=q^{-1}(q(V'))$. Since $q:\du{G}^X\to \du{H}^Y$ is a continuous surjective homomorphism
defined on the compact group $\du{G}^X$, the map $q$ is open. Since  $V'$ is an open subset of $\du{G}^X$, the set $U_V=q(V')$ is an open subset of $\du{H}^Y$.
\end{proof}

\section{\Round{n} sets and the powers of the dual group}
\label{Kurama:section}

\begin{lemma}
\label{denseness:in:countable:powers}
Assume that $H$ is an abelian group, $S\in\TT(H)$, $Y$ is a non-empty set and  $O$ is an open subset of $\T^Y$  such that  $O\cap \T[\n_S]^Y\not=\emptyset$. Then 
\begin{equation}
\label{dense:set:in:countable:powers}
P_{S,O}=\{\pi\in \du{H}^Y: \pi(S)\cap O\not=\emptyset\}
\end{equation}
is an open dense subset of  $\du{H}^Y$.
\end{lemma}

\begin{proof}
The set $P_{S,O}$ is open in  $\du{H}^Y=\Hom(H,\T^Y)$ by Lemma \ref{openess:of:sets}.  Let us prove that $P_{S,O}$ is dense in $\du{H}^Y$. We start with the case of a finite $Y$.
\begin{claim}
\label{finite:case:claim}
$P_{S,O}$ is dense in $\du{H}^Y$ when $Y$ is finite.
\end{claim}

\begin{proof} For $\n_S\ge 2$, the density of $P_{S,O}$ in $\du{H}^Y$ follows from \cite[Lemma 3.7]{DT}. By Remark \ref{Background}(i),
only the case $\n_S=0$ remains, so we  shall assume now that $\n_S=0$; that is, $S\in \TT_0(H)$.  The rest of the proof proceeds by induction on the size of the set $Y$.
If $|Y|=1$, then $P_{S,O}$ is dense in $\du{H}$ by \cite[Lemma 3.3]{TY}; see also \cite[Lemma 4.2]{DT}.  Let $k\in\N^+$, and assume that we have already proved that the set  $P_{S^*,O^*}$ is dense in $\du{H}^{Y^*}$ whenever $S^*\in\TT_0(H)$, $Y^*$ is a finite set with $1\le |Y^*|\le k$ and $O^*$ is a non-empty open subset of $\T^{Y^*}$. Suppose now that $S\in\TT_0(H)$,  $Y$ is a finite set of size $k+1$ and $O$ is a non-empty open subset of $\T^Y$. Let $V$ be a non-empty open subset of $\du{H}^Y$. It suffices to show that $V\cap P_{S,O}\not=\emptyset$.

Choose $y\in Y$ arbitrarily, and let $Y^*=Y\setminus\{y\}$. There exists an open subset $O^*$ of $\T^{Y^*}$ and an open subset $O'$ of $\T$ such that $O^*\times O'\subseteq O$. Similarly,  there exist a non-empty open subset $V^*$ of $\du{H}^{Y^*}$ and a non-empty open subset $V'$ of  $\du{H}$ such that $V^*\times V'\subseteq V$.
Since $S\in\TT_0(H)$, from Remark \ref{Background}(v)  it follows that  $\mathbb{D}(S,\mathbb{T})$ is dense in $\du{H}$, so we can pick $\pi'\in V'$ such that $\pi'(S)$ is dense in $\T$. 
Assume that  $S^*=\{x\in S: \pi'(x)\in O'\}$ is finite. Then $O''=O'\setminus\pi'(S^*)$ is a non-empty open subset of $\T$ such that $\pi'(S)\cap O''=\emptyset$, in contradiction with density of $\pi'(S)$ in $\T$. Therefore,  $S^*$ must be infinite. Since $S^*\subseteq S\in\TT_0(H)$, 
Remark \ref{Background}(iii) yields that $S^*\in\TT_0(H)$.  Applying our inductive assumption to $S^*$, $Y^*$ and $O^*$,
we can choose $\pi^*\in V^*\cap P_{S^*,O^*}$.  Then $\pi^*(S^*)\cap O^*\not=\emptyset$ by 
\eqref{dense:set:in:countable:powers}, so there exists $x\in S^*\subseteq S$
with $\pi^*(x)\in O^*$. Now $\pi=(\pi^*,\pi')\in V^*\times V'\subseteq V$ and $\pi(x)=(\pi^*(x),\pi'(x))\in O^*\times O'\subseteq O$. Therefore, $\pi(S)\cap O\not=\emptyset$. We proved that $\pi\in V\cap P_{S,O}\not=\emptyset$.
\end{proof}

Assume now that $Y$ is infinite. Choose any $z=\{z_y\}_{y\in Y}\in O\cap \T[\n_S]^Y$. 
Since $O$ is open in $\T^Y$, there exist a non-empty finite subset $F$ of $Y$  and an open subset $O_F$ of $\T^F$ such that $z\in O_F\times \T^{Y\setminus F}\subseteq O$.
 From this and 
\eqref{dense:set:in:countable:powers},
one gets $P_{S,O_F}\times \du{H}^{Y\setminus F}\subseteq P_{S,O}$. Since $z\in \T[\n_S]^Y$, we have $z_F=\{z_y\}_{y\in F}\in O_F\cap \T[\n_S]^F\not=\emptyset$.  By Claim \ref{finite:case:claim}, $P_{S,O_F}$ is dense in $\du{H}^F$,  and so 
$P_{S,O_F}\times \du{H}^{Y\setminus F}$ is dense in $\du{H}^Y$. Therefore,  $P_{S,O}$ must be  dense in $\du{H}^Y$ as well. 
\end{proof}

The main idea of the proof of our next lemma comes from the classical  proof    of the Hewitt-Marczewski-Pondiczery theorem; see, for example, \cite[Theorem 2.3.15]{Eng}. Essentially, we produce a sophisticated adaptation of that proof to the power $\du{H}^\cont$ of the dual group $\du{H}$. 

\begin{lemma}
\label{Kurama:lemma}
Let $H$ be an abelian group and $\mathscr{S}$ a countable subfamily of $\TT(H)$.
Then there exists a group homomorphism $\pi: H\to \T^\cont$ such that $\pi(S)$ is dense in $\T[\n_S]^\cont$ for every $S\in\mathscr{S}$.
\end{lemma}

\begin{proof} In this proof only, it is beneficial for us to follow the common set-theoretic practice of identifying a natural number $m\in\mathbb{N}$ with the set $\{0,\ldots,m-1\}$ of all its predecessors. In particular, $0=\emptyset$ and $m= \{0,\ldots,m-1\}$ for $m\in\N^+$.

Let $\B=\{B_l:l\in\mathbb{N}\}$ be a countable base of $\T$ such that $B_0=\T$.  Let $\mathscr{S}=\{S_j:j\in\mathbb{N}\}$ be an enumeration of $\mathscr{S}$. Define $n_j=\n_{S_j}$ for every $j\in\N$. For every $m\in\N^+$, let
\begin{equation}
\label{O:phi}
O_\mu=\prod_{g\in 2^m} B_{\mu(g)}
\ \ 
\mbox{ for every function }
\mu: 2^m\to m,
\mbox{ and }
\end{equation}
\begin{equation}
\label{eq:L_m}
L_m=\left\{(j,\mu): j=0,\dots,m-1, \mu: 2^m\to m
\mbox{ is a function and }
O_\mu\cap\T[n_j]^{2^m}\not=\emptyset\right\}.
\end{equation}

By induction on $m\in\mathbb{N}$ we define a family $\{U_g: g\in 2^m\}$ 
of non-empty open subsets of $\du{H}$ with the following properties:
\begin{itemize}
  \item[(i$_m$)] if $m\in\N^+$, then $\overline{U_g}\subseteq U_{g\restr{m-1}}$ for all $g\in 2^m$,
  \item[(ii$_m$)] if $m\in\N^+$, 
then $\prod_{g\in 2^m} U_g\subseteq \bigcap\{P_{S_j,O_\mu}: (j,\mu)\in L_m\}$.
\end{itemize}

$U_{\emptyset}=\du{H}$ trivially satisfies (i$_0$) and (ii$_0$). Suppose that $m\in\N^+$, and for every $k< m$ we have already defined a family $\{U_g: g\in 2^k\}$  of non-empty open subsets of $\du{H}$ 
satisfying (i$_k$) and (ii$_k$). We are going to define a family $\{U_g:g\in 2^m\}$ 
of non-empty open subsets of $\du{H}$ satisfying (i$_{m}$) and (ii$_{m}$).

For each $g\in 2^{m}$, since $U_{g\restr{m-1}}\not=\emptyset$ and the space $\du{H}$ is completely regular, we can choose a non-empty open subset $V_g$ of $\du{H}$ with $\overline{V_g}\subseteq U_{g\restr{m-1}}$. 
Since $L_m$ is finite, applying Lemma \ref{denseness:in:countable:powers} to $Y=2^m$, we conclude that 
$
P=\bigcap\{P_{S_j,O_\mu}: (j,\mu)\in L_m\}
$
is an open dense subset of $\du{H}^{2^m}$. Since $V=\prod_{g\in 2^m} V_g$ is a non-empty open subset of $\du{H}^{2^m}$, so is $P\cap V$. Therefore, there exists a family $\{U_g:g\in 2^m\}$ 
of non-empty open subsets of $\du{H}$ such that $\prod_{g\in 2^m} U_g\subseteq P\cap V$. By our construction,  $U_g\subseteq V_g\subseteq \overline{V_g}\subseteq U_{g\restr{m-1}}$ for all $g\in 2^m$, so (i$_m$) holds. Clearly, (ii$_m$) holds as well. This finishes the inductive construction.

Let $f\in 2^\N$. Since (i$_m$) holds for all $m\in\N$, we have 
$$
U_\emptyset=U_{f\restr{0}}
\supseteq 
\overline{U_{f\restr{1}}}\supseteq {U_{f\restr{1}}}
\supseteq 
\overline{U_{f\restr{2}}}\supseteq {U_{f\restr{2}}}
\supseteq 
\cdots
\supseteq 
\overline{U_{f\restr{m}}}\supseteq {U_{f\restr{m}}}
\supseteq 
\overline{U_{f\restr{m+1}}}\supseteq {U_{f\restr{m+1}}}
\supseteq
\cdots.
$$
Since all $U_{f\restr{m}}$ are non-empty and $\du{H}$ is compact, 
this yields
\begin{equation}
\label{eq:E:g}
C_f=\bigcap_{m\in\N} U_{f\restr{m}}
=
\bigcap_{m\in\N} \overline{U_{f\restr{m}}}
\not=\emptyset.
\end{equation}
Therefore, there exists $\pi_f\in C_f$.

Define a homomorphism $\pi:H\to \T^{2^\N}$  by $\pi(x)=\{\pi_f(x)\}_{f\in 2^\N}\in \T^{2^\N}$ for all $x\in H$. Since $\T^{2^\N}$ and $\T^\cont$ are topologically isomorphic, 
it remains only to prove that  $\pi(S)$ is dense in $\T[\n_S]^{2^\N}$ for every $S\in\mathscr{S}$. Fix $S\in\mathscr{S}$.  Then $S=S_j$ for some $j\in\N$.  Since $\n_S=\n_{S_j}=n_j$, it suffices to prove that $\pi(S_j)\cap O\not=\emptyset$ for every open subset $O$ of  $\T^{2^\N}$ such that $O\cap \T[n_j]^{2^\N}\not=\emptyset$.  Fix such an $O$. By the definition of the Tychonoff product topology,  there exist a finite set $F\subseteq 2^\N$ and a set  $\{l_f: f\in F\}\subseteq \N$, such that
\begin{itemize}
\item[(a)]   $\prod_{f\in F}B_{l_f}\cap \T[n_j]^F\not=\emptyset$, and 
\item[(b)]   $\prod_{f\in F} B_{l_f}\times \T^{2^\N\setminus F}\subseteq O$.
\end{itemize}

\begin{claim}
\label{m:and:xi}
There exist $m\in\N$ and a function $\varphi:2^m\to 2^\N$ such that:
\begin{itemize}
\item[(i)] $m>\max\{l_f:f\in F\}$ and $m>j$,
\item[(ii)] $\varphi(g)\restr{m}=g$ for all $g\in 2^m$,
\item[(iii)] $F\subseteq \varphi(2^m)$.
\end{itemize}
\end{claim}
\begin{proof}
Choose $m\in\N^+$ such that $f\restr{m}\not=f'\restr{m}$ whenever $f,f'\in F$ and $f\not=f'$. Without loss of generality, we 
may
 assume  that 
(i) holds.
Let  $G=\{f\restr{m}:f\in F\}$.
For $g\in G$, define $\varphi(g)$ to be the unique $f\in F$ with $g=f\restr{m}$.
For $g\in 2^m\setminus G$, let $\varphi(g)$ be an arbitrary $f\in 2^\N$ such that
$f\restr{m}=g$. Now (ii) and (iii) are satisfied.
\end{proof}

Claim \ref{m:and:xi}(i) allows us to 
define the function $\mu:2^m\to m$ 
by letting
$\mu(g)=l_{\varphi(g)}$ if 
$\varphi(g)\in F$
and $\mu(g)=0$ 
otherwise, for every $g\in 2^m$.

\begin{claim}
\label{subclaim:3}
$(j,\mu)\in L_m$.
\end{claim}

\begin{proof}
Note that $j<m$ by Claim \ref{m:and:xi}(i). According to \eqref{eq:L_m}, 
it remains only to check that $O_\mu\cap \T[n_j]^{2^m}\not=\emptyset$.
By \eqref{O:phi}, to accomplish this, it suffices to show that $B_{\mu(g)}\cap\T[n_j]\not=\emptyset$ for every  $g\in 2^m$. If $\varphi(g)\in F$, then  $\mu(g)=l_{\varphi(g)}$, and so  $B_{\mu(g)}\cap \T[n_j]=B_{l_{\varphi(g)}}\cap \T[n_j]\not=\emptyset$ by (a). If $\varphi(g)\not\in F$, 
then $\mu(g)=0$ and $B_{\mu(g)}\cap\T[n_j]=B_0\cap\T[n_j]=\T[n_j]\not=\emptyset$, as $B_0=\T$.  
\end{proof}
 
For each $g\in 2^m$, from Claim \ref{m:and:xi}(ii) and 
\eqref{eq:E:g}, we get $\pi_{\varphi(g)}\in C_{\varphi(g)}\subseteq U_{\varphi(g)\restr{m}}=U_g$.
Therefore, $\{\pi_{\varphi(g)}\}_{g\in 2^m}\in \prod_{g\in 2^m} U_g$.  From 
this,
Claim \ref{subclaim:3} and (ii$_m$), it follows that $\{\pi_{\varphi(g)}\}_{g\in 2^m}\in P_{S_j,O_\mu}$. Combining this with \eqref{dense:set:in:countable:powers} and \eqref{O:phi},  we can  select  $x\in S_j$ such that $\pi_{\varphi(g)}(x)\in B_{\mu(g)}$ whenever $g\in 2^m$.  

Let $f\in F$ be arbitrary.
It follows from items (ii) and (iii) of Claim \ref{m:and:xi}
that $f=\varphi(g)$, where
$g=f\restr{m}\in 2^m$,
and so
$\pi_f(x)=\pi_{\varphi(g)}(x) \in B_{\mu(g)}=B_{l_{\varphi(g)}}=B_{l_f}$. 
 From (b), we get $\pi(x)\in O$. Since $x\in S_j$, we obtain
$\pi(x)\in\pi(S_j)\cap O\not=\emptyset$.
\end{proof}

\section{Proof of Theorem \ref{main:result}}
\label{proof:section}

We are going to carry out the proof of item (i) and the ``if'' part of item (ii) simultaneously.
In order to do this, define $\Sigma^\flat_{G,\mathscr{E}}=\Sigma_{G,\mathscr{E}}\cap\Mono(G,\T^\cont)$ if  $|G|\le 2^\cont$ and $\Sigma^\flat_{G,\mathscr{E}}=\Sigma_{G,\mathscr{E}}$ otherwise. Our goal is to prove that $\Sigma^\flat_{G,\mathscr{E}}$ is a dense subspace of $\du{G}^\cont$ having the Baire property. By Remark \ref{Baire:remark}(i), in order to achieve this, 
it suffices to check that $\Sigma^\flat_{G,\mathscr{E}}\cap W\cap\bigcap\mathscr{V}^*\not=\emptyset$
whenever $W$ is a non-empty open subset of $\du{G}^\cont$ and $\mathscr{V}^*$ is a countable family of open dense subsets of $\du{G}^\cont$.

First, we 
apply Lemma \ref{reflection:lemma} to  $G$, $Z=\bigcup\E$, $X=\cont$  and $\mathscr{V}=\mathscr{V}^*\cup\{W\}$ to choose $H$, $Y$ and $\mathscr{U}$ as in the conclusion of this lemma. In particular, $\mathscr{U}=\{U_V:V\in\mathscr{V}\}$  is a family of open subsets of $\du{H}^Y$.

\begin{claim}
\label{claim:3}
$U_W\not=\emptyset$ and $U_V$ is a dense in $\du{H}^Y$ for every $V\in\mathscr{V}^*$. 
\end{claim}

\begin{proof}
Since the map $q^{XY}_{GH}$ is continuous, from  Lemma \ref{reflection:lemma} it follows that 
$U_V$ is dense in  $q^{XY}_{GH}(V)$ for every $V\in\mathscr{V}$.
In particular, $U_W$ is dense in $q^{XY}_{GH}(W)\not=\emptyset$, so $U_W$ must be non-empty. 
Since each $V\in\mathscr{V}^*$ is dense in $\du{G}^X$ and the map $q^{XY}_{GH}$ is continuous, $q^{XY}_{GH}(V)$ must be dense in $q^{XY}_{GH}(\du{G}^X)=\du{H}^Y$. Hence, $U_V$ is dense in $H^Y$
for every $V\in\mathscr{V}^*$.
\end{proof}

Let $\B$ be a countable base of $\T^Y$. Define
\begin{equation}
\label{eq:D}
D=\bigcap 
\left\{P_{E,O}: E\in\E, O\in\mathscr{B}, O\cap\T[\n_E]^Y\not=\emptyset\right\}\cap
\Mono(H,\T^Y)\cap \bigcap _{V\in\mathscr{V}^*} U_V,
\end{equation}
where $P_{E,O}$ are the sets defined in \eqref{dense:set:in:countable:powers}.

\begin{claim}
\label{claim:4}
There exists $\theta\in  D\cap U_W$.
\end{claim}

\begin{proof}
If $E\in\E$, $O\in\B$  and $O\cap\T[\n_E]^Y\not=\emptyset$, then  $P_{E,O}$ is a dense open subset of $\du{H}^Y$ by Lemma 
\ref{denseness:in:countable:powers}. Furthermore, $\Mono(H,\T^Y)$ is a dense $G_\delta$-subset of $\du{H}^Y$ by  Lemma \ref{monomorphisms:are:dense:G-delta:for:countable:groups}. For every $V\in\mathscr{V}^*$, $U_V$ is an open dense subset    of $\du{H}^Y$ by 
Claim \ref{claim:3}. Since $\E$, $\mathscr{B}$ and $\mathscr{V}^*$ are countable,  $D$ is an intersection of a countable family of open dense subsets of $\du{H}^Y$. Since $\du{H}^Y$ is compact, $D$ must be dense in $\du{H}^Y$. Since $U_W$ is a non-empty open subset of $\du{H}^Y$ by Claim \ref{claim:3}, $D\cap U_W\not=\emptyset$. This allows us to choose $\theta\in D\cap U_W$.
\end{proof}

For  $E\in\E$ and $O\in\B$,  let 
\begin{equation}
\label{eq:S}
S_{E,O}=\{x\in E: \theta(x)\in O\}.
\end{equation}
Since both $\E$ and $\B$ are countable, the set 
\begin{equation}
\label{eq:E'}
\mathscr{S}=\{S_{E,O}: E\in\E, O\in\B, |S_{E,O}|=\omega\}.
\end{equation}
is countable as well. 

\begin{claim}
\label{claim:pi}
There exists a homomorphism $\pi:H\to \T^{\cont\setminus Y}$ such that 
$\pi(S)$ is dense in $\T[\n_S]^{\cont\setminus Y}$ for every $S\in\mathscr{S}$.
\end{claim}
\begin{proof}

Let $S\in \mathscr{S}$. Then $S=S_{E,O}\subseteq E\subseteq \bigcup\mathscr{E}=Z\subseteq H$ for some $E\in\E$ and $O\in\B$.
Since $\E\subseteq \TT(G)$ and $S$ is infinite, from this and  Remark \ref{Background}(iii) we conclude that $\mathscr{S}\subseteq \TT(G)$.
Since $S\subseteq H$ for every $S\in\mathscr{S}$, we also have $\mathscr{S}\subseteq \TT(H)$ by Remark \ref{Background}(iv).
Since $Y$ is countable,  $|\cont\setminus Y|=\cont$, and so $\T^\cont$ and $\T^{\cont\setminus Y}$ are topologically isomorphic. Now the conclusion follows from Lemma \ref{Kurama:lemma}.
\end{proof}

Since $\theta\in D\subseteq \Mono(H,\T^Y)$ by Claim \ref{claim:4} and \eqref{eq:D}, $\theta: H\to\T^Y$ is a monomorphism. Therefore,   the map $\xi=(\theta,\pi): H\to \T^Y\times\T^{\cont\setminus Y}$ is a monomorphism as well.  Since $\T^\cont$ is divisible, there exists a  homomorphism $\sigma:G\to \T^Y\times\T^{\cont\setminus Y}=\T^\cont$ extending $\xi$. Furthermore, if $|G|\le 2^\cont$,  then \cite[Lemma 3.17]{DS} allows us to find a monomorphism $\sigma:G\to \T^Y\times\T^{\cont\setminus Y}=\T^\cont$ extending $\xi$. Indeed, the assumptions of \cite[Lemma 3.17]{DS} are satisfied, because $\T^\cont$ is divisible, $|G|\le 2^\cont$, $|H|=\omega<2^\cont$, 
$|\T^\cont|=r(\T^\cont)=2^\cont$ and $r_p(\T^\cont)=2^\cont$ for all prime numbers $p$ (see, for example, \cite[Lemma 4.1]{DS-Memoirs} for computations of the ranks of $\T^\cont$).

Let us show that $\sigma\in \Sigma_{G,\E}^\flat\cap W\cap\bigcap\mathscr{V}^*$,
thereby proving that $\Sigma_{G,\E}^\flat\cap W\cap\bigcap\mathscr{V}^*\not=\emptyset$. Note that 
$$
q^{XY}_{GH}(\sigma)=\theta\in D\cap U_W\subseteq  \bigcap \{U_V: V\in\mathscr{V}^*\}\cap U_W 
=
\bigcap\{U_V:V\in\mathscr{V}\}
$$
by Claim \ref{claim:4} and \eqref{eq:D}, so
$$
\sigma\in \left(q^{XY}_{GH}\right)^{-1}\left(\bigcap_{V\in\mathscr{V}} U_V\right)=
\bigcap_{V\in\mathscr{V}}\left(q^{XY}_{GH}\right)^{-1}(U_V)
\subseteq 
\bigcap_{V\in\mathscr{V}}V
=
\bigcap\mathscr{V}
=
W\cap \bigcap\mathscr{V}^*
$$
by Lemma \ref{reflection:lemma}. Since $\sigma\in\Mono(G,\T^\cont)$ when $|G|\le 2^\cont$, it remains only to check that $\sigma\in  \Sigma_{G,\E}$. Fix $E\in\E$. Recalling  \eqref{Sigma:G}, we need to prove that $\sigma(E)$ is dense in $\T[\n_E]^\cont$. Let $O^*$ be an arbitrary open subset of $\T^\cont$ with  $O^*\cap \T[\n_E]^\cont\not=\emptyset$. It suffices to prove that $\sigma(E)\cap O^*\not=\emptyset$. Since $\mathscr{B}$ is a base of $\T^Y$, there exist $O\in\B$ and an open subset $V$ of $\T^{\cont\setminus Y}$ such that $O\times V\subseteq O^*$, $O\cap \T[\n_E]^Y\not=\emptyset$ and $V\cap \T[\n_E]^{\cont\setminus Y}\not=\emptyset$. 

\begin{claim}
\label{subclaim:6.1}
$S_{E,O}\in\mathscr{S}$.
\end{claim}

\begin{proof} By 
\eqref{eq:E'}, it suffices to show that the set $S_{E,O}$ is infinite. Recall that $\n_E\not=1$ by Remark \ref{Background}(i). Since $Y$ is infinite, the  non-empty open subset $O\cap \T[\n_E]^Y$ of  $\T[\n_E]^Y$ is infinite. Assume that $S_{E,O}$ is finite. Then $O'=O\setminus\theta(S_{E,O})$
is a non-empty open subset of $\T^Y$ with $O'\cap \T[\n_E]^Y\not=\emptyset$. Since $\mathscr{B}$ is a base of $\T^Y$, we can choose $O''\in\mathscr{B}$ such that $O''\subseteq O'$ and $O''\cap \T[\n_E]^Y\not=\emptyset$. Then $\theta\in D\subseteq P_{E,O''}$ by \eqref{eq:D}, and  from \eqref{dense:set:in:countable:powers} we conclude that $\theta(x)\in O''\subseteq O'$ for some $x\in E$.
Since $O'\subseteq O$,  from \eqref{eq:S} it follows that $x\in S_{E,O}$, which yields $\theta(x)\in \theta(S_{E,O})$. Therefore, $\theta(x)\not\in O\setminus \theta(S_{E,O})=O'$, giving a contradiction.
\end{proof}

\begin{claim}
\label{subclaim:6.2}
$\n_{S_{E,O}}=\n_E$.
\end{claim}
\begin{proof}
Being an infinite subset of  the \round{\n_E} set $E$, the set $S_{E,O}$ is also \round{\n_E} by Remark \ref{Background}(iii). Now the conclusion follows from Remark \ref{Background}(ii).
\end{proof}

The set $\pi(S_{E,O})$ is dense in $\T[\n_E]^{\cont\setminus Y}$ by Claims \ref{claim:pi}, \ref{subclaim:6.1} and \ref{subclaim:6.2}. Since $V\cap \T[\n_E]^{\cont\setminus Y}$ is a non-empty open subset of $\T[\n_E]^{\cont\setminus Y}$, there exists $x\in S_{E,O}$ with $\pi(x)\in V$. Therefore, $\theta(x)\in O$ by \eqref{eq:S}, and so
$\xi(x)=(\theta(x),\pi(x))\in O\times V\subseteq O^*$.  Since $x\in S_{E,O}\subseteq E$, it follows that $\xi(x)\in \xi(E)\cap O^*\not=\emptyset$. 
Since $x\in E\subseteq Z\subseteq H$ and $\sigma\restr{H}=\xi$, we conclude  that $\sigma(x)\in \sigma(E)\cap O^*\not=\emptyset$. 
This finishes the proof of item (i) and the ``if'' part of item (ii).

Assume now that $\Sigma_{G,\mathscr{E}}\cap\Mono(G,\T^\cont)$ is dense in $\du{G}^\cont$,
and choose $\sigma\in\Sigma_{G,\mathscr{E}}\cap\Mono(G,\T^\cont)$. Since $\sigma$ is a monomorphism, $|G|=|\sigma(G)|\le |\T^\cont|\le 2^\cont$. This proves the ``only if'' part of item (ii).
\qed

\section{Comparison of Borel complexity of $\mathbb{D}({S,K})$ and $\mathbb{U}({S,K})$}
\label{section:9}

\begin{prop} 
\label{G:delta-sets}
 Let $G$ be an abelian group and $K$ a compact metric abelian group.
\begin{itemize}
\item[(i)] If $S$ is a countably infinite subset of $G$, then $\mathbb{D}({S,K})$ is a $G_{\delta}$-set in $\Hom(G,K)$. 
\item[(ii)] If $S=\{a_n:n\in\N\}$ is a  one-to-one sequence in $G$, then $\mathbb{U}({S,K})$ is an $F_{\sigma\delta}$-set in $\Hom(G,K)$. 
\end{itemize}
\end{prop}

\begin{proof} 
(i) Fix a countable base $\mathscr{B}$ of  $K$ such that $\emptyset\not\in \mathscr{B}$.  For every $O\in \mathscr{B}$,  the set $P_{S,O}=\{\pi\in \Hom(G,K): \pi(S)\cap O\not=\emptyset\}$ is open in  $\Hom(G,K)$ by Lemma \ref{openess:of:sets} (applied to $G$ instead of $H$). Since $\mathbb{D}({S,K})=\bigcap \{P_{S,O}: O\in\mathscr{B}\}$, 
it follows that $\mathbb{D}({S,K})$ is a $G_{\delta}$-set in $\Hom(G,K)$. 

(ii) Since $K$ is a compact metric group, the dual group $\du{K}$ is countable. Let $\{\chi_k:k\in\N\}$ be an enumeration of all non-trivial continuous $\mathbb{S}$-characters of $K$. For $k\in\N$ and $m,n\in\N^+$, 
$$
F_{k,m,n}=\left\{\pi\in \Hom(G,K):  \left|\frac{1}{n}\sum_{i=1}^{n}\chi_k(\pi(a_i))\right|\leq  \frac{1}{m}\right\}
$$
is a closed subset of  $\Hom(G,K)$, and so $\bigcap_{n=j}^\infty F_{k,m,n}$ is also a closed subset of  $\Hom(G,K)$ for every  $j\in \N^+$. Combining this with Weyl's criterion, we conclude that
$$
\mathbb{U}({S,K})=\bigcap_{k\in\N}\bigcap_{m\in\N^+}\bigcup_{j\in \N^+}\left(\bigcap_{n=j}^\infty F_{k,m,n}\right)
$$
is an $F_{\sigma\delta}$-subset of $\Hom(G,K)$. 
\end{proof}

\begin{cor}
For every compact metric abelian group $K$ and each strictly increasing sequence $S=\{a_n:n\in\N\}$ of integers,  the set $\mathrm{Weyl}({S,K})$ is an $F_{\sigma\delta}$-set in $K$.
\end{cor}

\begin{proof}
Follows from Proposition \ref {G:delta-sets}(ii) and the observation made after Definition \ref{definition:U:D}.
\end{proof}

The authors do not know of any example of a strictly increasing sequence $S=\{a_n:n\in\N\}$ of integers and a compact metric group $K$ for which $\mathrm{Weyl}({S,K})$ (and thus, $\mathbb{U}({S,K})$ as well) is not a $G_\delta$-set. Of course, the case $K=\T$ is the most interesting here.

Item (iii) of our next remark
shows that  Proposition \ref {G:delta-sets}(i) does not hold for a non-metric group $K$.

\begin{remark}
\label{non:G-delta}
\begin{itemize}
\item[(i)] 
{\em Let $E$ be a subset of an abelian group $G$, $\kappa$ an uncountable
cardinal and $n\in\N\setminus\{1\}$. Then the set $\Pi(E,\kappa,n)=\{\pi\in\Hom(G,\T^\kappa): \pi(E)$ is dense in  $\T[n]^\kappa\}$
does not contain any non-empty $G_\delta$-subset $B$ of $\Hom(G,\T^\kappa)$.\/} Indeed, given such a $B$, there exist a countable subset $Y$ of $\kappa$
and $\pi=\{\pi_\alpha\}_{\alpha< \kappa}\in B$ such that $\pi_\alpha=0$ for all $\alpha\in \kappa\setminus Y$.  Let $\beta\in \kappa\setminus Y$ be arbitrary. Since $n\neq 1$, $\T[n]\not=\{0\}$. Since $\pi_{\beta}(E)=\{0\}$, the set $\pi_{\beta}(E)$ is not dense in $\T[n]$. Therefore, $\pi(E)$ cannot be dense in $\T[n]^\kappa$.
Hence, $\pi\in B\setminus \Pi(E,\kappa,n)$. 

\item[(ii)]
{\em The family  $\Sigma_{G,\E}$ from 
Theorem \ref{main:result} does not contain any non-empty $G_\delta$-subset of $\du{G}^\cont$\/}. Indeed, since $\E\not=\emptyset$, we can choose some $E\in\E\subseteq \TT(G)$. Then $\n_E\neq 1$ by Remark \ref{Background}(i). Finally, note that $\Sigma_{G,\E}\subseteq \Pi(E,\cont,\n_E)$ and apply item (i).

\item[(iii)]
{\em If $S$ is a subset of an abelian group $G$ and $X$ is an uncountable set, then $\mathbb{D}(S,\T^X)$ does not contain any non-empty $G_\delta$-subset of $\du{G}^X$\/}.
Indeed, observe that  $\mathbb{D}(S,\T^X) =\Pi(S,|X|,0)$ and apply item (i).
\end{itemize}
\end{remark}


\begin{thebibliography}{99}

\bibitem{BMR} G.~Baumslag, A.~Myasnikov, 
V.~Remeslennikov, 
Algebraic geometry over groups, I. Algebraic sets and ideal theory,
J. Algebra 219 (1999) 16--79. 

\bibitem{B}
H.~Bohr, 
Collected Mathematical Works,
E.~F\o elner and B.~Jessen, eds.
Dansk Matemetitisk Forsening, 1952. 

\bibitem{Bryant} R.~Bryant,  
The verbal topology of a group,
J. Algebra 48 (1977) 340--346.

\bibitem{CSa} 
W.~Comfort, V.~Saks, 
Countably compact groups and finest totally bounded topologies,
Pacific J. Math. 49 (1973) 33--44.

\bibitem{Trigos} W.~Comfort, J.~Trigos-Arrieta, 
Remarks on a theorem of Glicksberg,
in:
General topology and applications (Staten Island, NY, 1989), Lecture Notes in Pure and Appl. Math., 134, Dekker, New York, 1991, pp.~25--33.

\bibitem{DS-Memoirs}
D.~Dikranjan, D.~Shakhmatov,
Algebraic structure of  pseudocompact groups,
Memoirs Amer. Math. Soc., 133/633 (1998), 83 pp.

\bibitem{DS}
D.~Dikranjan, D.~Shakhmatov,
Forcing hereditarily separable compact-like group topologies on Abelian groups,
Topology Appl. 151  (2005) 
2--54. 

\bibitem{DS_HMP} D.~Dikranjan, D.~Shakhmatov, 
Hewitt-Marczewski-Pondiczery type theorem for abelian groups and Markov's potential density,
Proc. Amer. Math. Soc. 138 (2010) 2979--2990.

\bibitem{DS_JGT} D.~Dikranjan, D.~Shakhmatov, 
Reflection principle characterizing groups in which unconditionally closed sets are algebraic,
 J. Group Theory 11 (2008) 421--442.

\bibitem{DS_OPIT} D.~Dikranjan, D.~Shakhmatov, 
Selected topics from the structure theory of topological groups,
in: Open Problems in Topology II (E.~Pearl, ed.), Elsevier, 2007, pp. 389--406. 

\bibitem{DS-MZ}
D.~Dikranjan, D.~Shakhmatov, 
The Markov-Zariski topology of an abelian group,
J. Algebra 324 (2010) 1125--1158.

\bibitem{DT}
D.~Dikranjan, M.~Tkachenko, 
Algebraic structure of small countably compact Abelian groups,
Forum Math. 15 (2003) 811--837.

\bibitem{DT0} D.~Dikranjan, M.~Tkachenko, 
Weakly complete free topological groups,
Topology Appl. 112 (2001) 259--287.

\bibitem{vD} 
E.\,K.~van~Douwen, 
The maximal totally bounded group topology on $G$ and the biggest minimal $G$-space for Abelian groups $G$,
Topology Appl. 34 (1990) 69--91.

\bibitem{Eng} R.~Engel'king, 
General topology
(Second edition), Sigma Series in Pure Mathematics, 6, Heldermann Verlag, Berlin, 1989. 

\bibitem{Flor} 
P.~Flor, 
Zur Bohr-Konvergenz von Folgen,
Math. Scand. 23 (1968) 169--170.

\bibitem{Fu} 
L.~Fuchs, 
Infinite Abelian groups,
Vol. I, Academic Press, New York, 1970.

\bibitem{Glicksberg} 
I.~Glicksberg, 
Uniform boundedness for groups,
Canad. J. Math. 14 (1962) 269--276.

\bibitem{Hewitt}
E.~Hewitt, 
Rings of real-valued continuous functions I,
Trans. Amer. Math. Soc. 64 (1948) 45--99. 

\bibitem{Mar1} A.\,A.~Markov, 
On unconditionally closed sets,
Comptes Rendus Dokl. AN SSSR (N.S.) 44 (1944) 180--181 (in Russian).

\bibitem{Mar} A.\,A.~Markov, 
On unconditionally closed sets,
Mat. Sbornik 18 (1946), 3--28 (in Russian). English translation in:
Markov,~A.\,A. Three papers on topological groups:~I. On the existence of periodic connected topological groups,~II. On free topological groups,~III. On unconditionally closed sets. 
Amer. Math. Soc. Translation 1950, (1950). no.~30, 120~pp.;
yet another English translation in:
Topology and Topological Algebra, Translations Series~1, vol.~8,
pp.~273--304. American Math. Society, 1962.

\bibitem{Pospishil}
B.~Posp\'{\i}\v sil, 
Sur la puissance d'un espace contenant une partie dense de puissance
donn\'ee,
\v{C}asopis Pro P\v{e}stov\'an\'{\i} Matematiky a Fysiky 67 (1937) 89--96.

\bibitem{TY}
M.~Tkachenko, I.~Yaschenko,
Independent group topologies on Abelian groups,
Topology Appl. 122 (2002) 425--451.

\bibitem{Weil} 
A.~Weil, 
Sur les espace \' Structure Uniforme et sur la Topologie G\'en\' erale,
Hermann et Cie Paris, Publ. Math. Institute. Strasburg, 1937. 

\bibitem{W} H.~Weyl, 
\" Uber die Gleichverteilung von Zahlen mod. Eins,
Math. Ann. 77, no. 3 (1916) 313--352. 
\end{thebibliography}
\end{document}